\documentclass{article}
\usepackage[utf8]{inputenc}
\usepackage{tikz}
\usepackage{amsmath}
\usepackage{amsthm}

\newtheorem{theorem}{Theorem}[section]
\newtheorem{corollary}{Corollary}[theorem]
\newtheorem{lemma}[theorem]{Lemma}

\title{Counting Graph Homomorphisms Involving Complete Graphs}
\author{Jeffrey Beyerl, Cameron Sharpe}
\date{May 5, 2017}

\begin{document}

\maketitle

\begin{center}
\textbf{Abstract}
\end{center}

In the branch of mathematics known as graph theory, graphs are considered as a set of points, called \textit{vertices}, with connections between these points, called \textit{edges}. The purpose of this paper is to study mappings between two graphs that have certain desirable properties, called \textit{graph homomorphisms}, and the probability of such a mapping occurring. By using notions from graph theory and combinatorics, in this paper we prove several new theorems that place bounds on this probability for  certain common classes of graphs such as $K_n$, and show that isolated vertices may safely be ignored.

\section{Introduction}
For any graph $G$, we denote the set of its vertices as $V(G)$ and the set of edges as $E(G)$. For the purposes of this paper, all graphs are assumed to be \textit{simple, undirected} graphs. That is, each pair of vertices may have at most one edge associated with it, and such a pair of vertices are called the \textit{endpoints} of the edge. All edges are uniquely identified using their endpoints, so ${v_1}{v_2}$ specifies an edge between the vertices $v_1$ and $v_2$. The \textit{degree} of a vertex $v$, denoted $deg(v)$, is defined as the number of edges that have $v$ as an endpoint. If $deg(v)=0$, we say that $v$ is an \textit{isolated vertex}. Two vertices are considered \textit{adjacent} if there is an edge between them. A set of mutually non-adjacent vertices is called a \textit{stable set} (or sometimes an \textit{independent set}) of the graph. Special classes of graphs include the \textit{complete} graph, in which there are no stable sets, as well as the \textit{path} and \textit{cycle} graphs, examples of which are shown in Figure 1. 

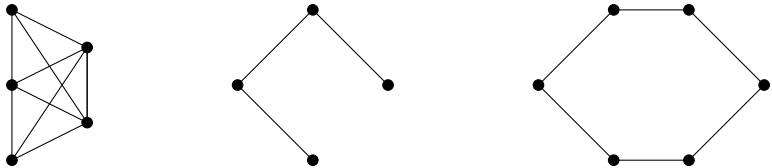
\begin{figure}[h]
    \centering
    \begin{tikzpicture}
\filldraw[black] (-5,-1) circle (2pt);
\filldraw[black] (-5,0) circle (2pt);
\filldraw[black] (-5,1) circle (2pt);
\filldraw[black] (-4,.5) circle (2pt);
\filldraw[black] (-4,-.5) circle (2pt);
\draw[black] (-5,-1) -- (-5,0) -- (-5,1) -- (-4,.5) -- (-4,-.5) -- cycle;
\draw[black] (-5,-1) -- (-4,.5);
\draw[black] (-5,1) -- (-4,-.5);
\draw[black] (-5,0) -- (-4,.5) -- (-4,-.5) -- cycle;

\filldraw[black] (0,0) circle (2pt);
\filldraw[black] (-1,1) circle (2pt);
\filldraw[black] (-2,0) circle (2pt);
\filldraw[black] (-1,-1) circle (2pt);
\draw[black] (-1,-1) -- (-2,0) -- (-1,1) -- (0,0);

\filldraw[black] (2,0) circle (2pt);
\filldraw[black] (3,1) circle (2pt);
\filldraw[black] (4,1) circle (2pt);
\filldraw[black] (5,0) circle (2pt);
\filldraw[black] (4,-1) circle (2pt);
\filldraw[black] (3,-1) circle (2pt);
\draw[black] (2,0)--(3,1)--(4,1)--(5,0)--(4,-1)--(3,-1)--cycle;
    \end{tikzpicture}
    \caption{From left to right: the complete graph on five vertices, $K_5$; the path graph on four vertices, $P_4$; and the cycle graph on six vertices, $C_6$}
    \label{fig:1}
\end{figure}

Given two graphs, $G$ and $F$, a \textit{graph homomorphism} from $G$ to $F$ is defined as a mapping $\varphi:G\rightarrow F$ such that ${v_1}{v_2}\in E(G)$ implies ${\varphi(v_1)}{\varphi(v_2)}\in E(F)$. That is, a graph homomorphism is a mapping in which edges are preserved. For more information on graph theory, see any text on the subject, such as \cite{golumbic}. Throughout this paper, $M$ will represent the set of all possible mappings from $G$ to $F$; $I$ will be those mappings which are injective, while $H$ will represent the set of all such mappings that are homomorphisms. The primary object of study for this paper is the \textit{homomorphism density}, defined as the probability of a homomorphism occurring between two graphs, and denoted by $t(G,F)=\frac{|H|}{|M|}$, first encountered in \cite{diaconis}. Over the course of our investigation, theorems have been produced that characterize when $t(G,F)=1$, as well as provide a set of inequalities pertaining to the complete graphs $K_n$. We also prove that the addition of an isolated vertex to the domain does not alter the value of $t(G,F)$. For a more comprehensive presentation of graph homomorphisms, please see \cite{godsil}.
\section{Main Theorems}
We begin our investigation of the homomorphism density with the following question: under what circumstances does $t(G,F)=1$? That is, when is a function from $G$ to $F$ guaranteed to be a homomorphism? It is clear that, if $G$ has no edges to preserve, then every mapping $\varphi:G\rightarrow F$ will be a graph homomorphism. Less clear is whether this is the \textit{only} way to guarantee every mapping from $G$ to $F$ is a homomorphism. However, if we consider any two graphs, $G$ and $F$, where $G$ has at least one edge, then we can easily obtain a mapping $\varphi:G\rightarrow F$ such that two endpoints of an edge in $G$ map to the same vertex in $F$. Since we require all graphs to be simple, this mapping will fail to be a homomorphism, and so we have established our first theorem:
\begin{theorem}
For two graphs $F$ and $G$, $t(F,G)=1$ if and only if $F=(K_n)^c$ for some natural number $n$.
\end{theorem}
\par With this theorem in mind, it is natural to suspect that there are cases for which $t(G,F)=0$, and indeed this is the case. Consider the example of $K_4$ and $P_3$, as shown in Figure 2. Since $F$ has only three vertices, at least two of the vertices from $G$ must map to the same vertex in $F$ for any $\varphi:G\rightarrow F$. Therefore, there are no possible homomorphisms from $G$ to $F$, and so $t(G,F)=0$.

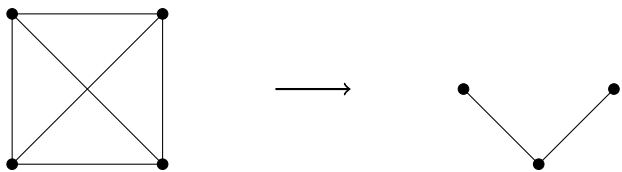
\begin{figure}[h]
    \centering
    \begin{tikzpicture}
        \filldraw[black] (-4,1) circle (2pt);
        \filldraw[black] (-4,-1) circle (2pt);
        \filldraw[black] (-2,1) circle (2pt);
        \filldraw[black] (-2,-1) circle (2pt);
        \draw[black] (-4,1) -- (-2,1) -- (-2,-1) -- (-4,-1) -- cycle;
        \draw[black] (-4,1) -- (-2,-1);
        \draw[black] (-2,1) -- (-4,-1);
        \draw[black,thick,->] (-.5,0) -- (.5,0);
        \filldraw[black] (2,0) circle (2pt);
        \filldraw[black] (3,-1) circle (2pt);
        \filldraw[black] (4,0) circle (2pt);
        \draw[black] (2,0) -- (3,-1) -- (4,0);
    \end{tikzpicture}
    \caption{There is no mapping from $K_4$ on the left to $P_3$ that will preserve all the edges of $K_4$.}
    \label{fig:2}
\end{figure}

After establishing that we can make $t(G,F)$ equal to one or zero, the next question is whether this probability can be made arbitrarily close to those bounds. We found this to be a very hard problem to tackle, as graphs in general may have an astonishing array of vertices and edges, even with the restriction to simple and undirected graphs. Thus, for our purposes, we chose to focus on homomorphism densities involving the class of complete graphs. To do this, we started with a general graph $G$ and let $F$ be the complete graph on $n$ vertices, $K_n$. In investigating the relationship between injective functions and homomorphisms in such a case, the following two lemmas were developed:

\begin{lemma}
Let $G$ be any graph. If $\varphi:K_n \rightarrow G$ is a homomorphism, then $\varphi$ is injective.
\end{lemma}

\begin{proof}
Suppose $\varphi:K_n \rightarrow G$ is a homomorphism such that $\varphi$ is not injective. Then there exist vertices $u,v\in V(K_n)$ such that $\varphi(u)=\varphi(v)$. Since $K_n$ is complete, $uv\in E(K_n)$ for every $u,v\in V(K_n)$. However, since G is simple, $\varphi(u)\varphi(v)=\varphi(u)\varphi(u) \not\in E(G)$. Thus, $\varphi$ is not a homomorphism and a contradiction is reached.
\end{proof}

\begin{lemma}
Let $G$ be any graph. If $\varphi:G\rightarrow K_n$ is injective, then $\varphi$ is a homomorphism.
\end{lemma}

\begin{proof}
Assume $\varphi:G\rightarrow K_n$ is injective, then $\varphi(u)\neq \varphi(v)$ for any $u\neq v \in V(G)$. Since $K_n$ is complete, $\varphi(u)\varphi(v)\in E(K_n)$ for every $u\neq v \in V(G)$. Therefore, $uv\in E(G) \Rightarrow \varphi(u)\varphi(v)\in E(K_n)$ and so it must be the case that $\varphi$ is a homomorphism.
\end{proof}

From Lemma 2.2, we know that, for any complete graph $K_n$ and any fixed but arbitrary graph $F$, all homomorphisms from $K_n$ to $F$ must also be injective. That is, $H\subseteq I$. In particular, this means that there must always be at least as many injective mappings from $K_n$ to $F$ as there are homomorphisms, and so our next theorem is established:

\begin{theorem}
Let $n$ be a natural number $n$, and let $F$ be a fixed, arbitrary graph. Then $t(K_n,F)\leq\frac{|I|}{|M|}$.
\end{theorem}
An illustration of this theorem is provided in Figure 3.
\begin{figure}[h]
    \centering
    \begin{tikzpicture}
        \filldraw[black] (-2,1) circle (2pt);
        \filldraw[black] (-3,0) circle (2pt);
        \filldraw[black] (-2,-1) circle (2pt);
        \draw[black] (-2,1) -- (-3,0) -- (-2,-1) -- cycle;

        \draw[black,->] (-.5,0) -- (.5,0);

        \filldraw[black] (2,0) circle (2pt);
        \filldraw[black] (3,.5) circle (2pt);
        \filldraw[black] (4,-.5) circle (2pt);
        \filldraw[black] (5,0) circle (2pt);
        \filldraw[black] (3.5,-.5) circle (2pt);
        \filldraw[black] (4.5,.5) circle (2pt);
        \draw[black] (2,0) -- (3,.5) -- (4,-.5) -- (5,0);
        \draw[black] (3,.5) -- (3.5,-.5) -- (4,-.5);
        \draw[black] (4,-.5) -- (4.5,.5);
    \end{tikzpicture}
    \caption{For $G=K_3$ on the left and mapping into the graph $F$ on the left, we have that $|I|=120, |M|=216$. Therefore, by Theorem 2.4, we know that $t(G,F)\leq\frac{120}{216}=\frac{5}{9}$.}
    \label{fig:3}
\end{figure}
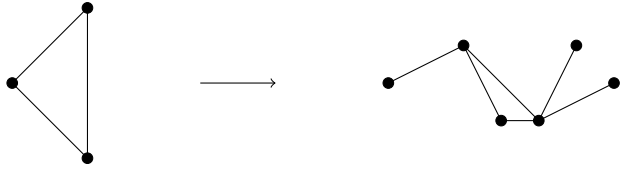

Using Lemma 2.3 and similar logic regarding injective mappings and homomorphisms, we obtain a similar theorem pertaining to when $G$ is our domain and $K_n$ the codomain:

\begin{theorem}
Let $G$ be a fixed but arbitrary graph, and let $n$ be a natural number. Then $t(G,K_n)\geq \frac{|I|}{|M|}$.
\end{theorem}
For an example of Theorem 2.5, refer to Figure 4 below.
\begin{figure}[h]
    \centering
    \begin{tikzpicture}
        \draw[black,->] (-.5,0) -- (.5,0);

        \filldraw[black] (-2,0) circle (2pt);
        \filldraw[black] (-3,.5) circle (2pt);
        \filldraw[black] (-4,-.5) circle (2pt);
        \filldraw[black] (-5,0) circle (2pt);
        \filldraw[black] (-3.5,-.5) circle (2pt);
        \filldraw[black] (-4.5,.5) circle (2pt);
        \draw[black] (-2,0) -- (-3,.5) -- (-4,-.5) -- (-5,0);
        \draw[black] (-3,.5) -- (-3.5,-.5);
        \draw[black] (-4,-.5) -- (-4.5,.5);
        
        \filldraw[black] (2,0) circle (2pt);
        \filldraw[black] (3,1) circle (2pt);
        \filldraw[black] (5,1) circle (2pt);
        \filldraw[black] (6,0) circle (2pt);
        \filldraw[black] (5,-1) circle (2pt);
        \filldraw[black] (3,-1) circle (2pt);

        \draw[black] (2,0)--(3,1)--(5,1)--(6,0)--(5,-1)--(3,-1)--cycle;
        \draw[black] (2,0)--(5,1)--(5,-1)--cycle;
        \draw[black] (6,0)--(3,1)--(3,-1)--cycle;
        \draw[black] (2,0)--(6,0);
        \draw[black] (3,1)--(5,-1);
        \draw[black] (3,-1)--(5,1);
    \end{tikzpicture}
    \caption{For the graphs $G$ on the left and $F=K_6$ on the right, we have that $|I|=720$ and $|M|=46,656$. Thus, by Theorem 2.5, we know that $t(G,F)\geq\frac{720}{46656}=\frac{5}{324}$.}
    \label{fig:4}
\end{figure}
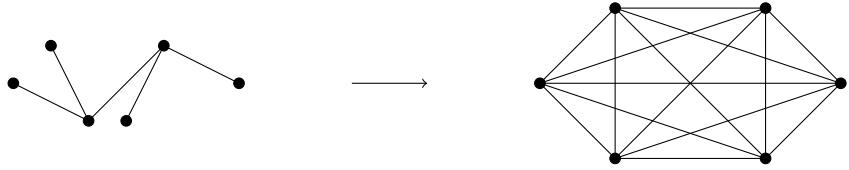

By combining Theorems 2.4 and 2.5, we obtain the following corrollary, addressing the case when both $G$ and $F$ are complete graphs.

\begin{corollary}
Let $G$ be the complete graph on $n$ vertices, $K_n$, and $F$ the complete graph on $m$ vertices, $K_m$. Then $t(K_n,K_m)=\frac{|I|}{|M|}$.
\end{corollary}
As an example of the corrollary, consider the case when $G=K_4$ and $F=K_5$, illustrated in Figure 5.

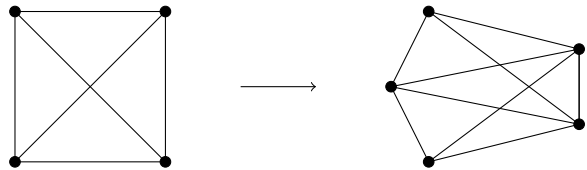
\begin{figure}[h]
    \centering
    \begin{tikzpicture}
        \filldraw[black] (-3.5,1) circle (2pt);
        \filldraw[black] (-3.5,-1) circle (2pt);
        \filldraw[black] (-1.5,1) circle (2pt);
        \filldraw[black] (-1.5,-1) circle (2pt);
        \draw[black] (-3.5,1) -- (-1.5,1) -- (-1.5,-1) -- (-3.5,-1) -- cycle;
        \draw[black] (-3.5,1) -- (-1.5,-1);
        \draw[black] (-1.5,1) -- (-3.5,-1);
        \draw[black,->] (-.5,0) -- (.5,0);
        
        \filldraw[black] (2,-1) circle (2pt);
        \filldraw[black] (1.5,0) circle (2pt);
        \filldraw[black] (2,1) circle (2pt);
        \filldraw[black] (4,.5) circle (2pt);
        \filldraw[black] (4,-.5) circle (2pt);
        \draw[black] (2,-1) -- (1.5,0) -- (2,1) -- (4,.5) -- (4,-.5) -- cycle;
        \draw[black] (2,-1) -- (4,.5);
        \draw[black] (2,1) -- (4,-.5);
        \draw[black] (1.5,0) -- (4,.5) -- (4,-.5) -- cycle;
    \end{tikzpicture}
    \caption{Here, $|I|=120$ while $|M|=625$. Thus, applying Corollary 2.5.1, we know that $t(G,F)=\frac{120}{625}=\frac{24}{125}.$}
    \label{fig:5}
\end{figure}
While the above theorems are quite useful in determining the bounds of $t(G,F)$ when either $G$ or $F$ is a complete graph, they do little to assist us when we consider graphs in general. This is due, in part, to the difficulty of exactly counting individual graph homomorphisms, which, when properly phrased, has been shown to be an NP-hard problem. More detail on the complexity of this problem can be found in \cite{dyer}. To tackle the problem of bounding $t(G,F)$ for general graphs, we first examined how the addition of a single, isolated vertex to $G$ might affect $t(G,F)$. This resulted in Theorem 2.6, stated below with proof.

\begin{theorem}
For a fixed graph $F$, if $G_0$ and $G_1$ are graphs such that\\ $V(G_1)=V(G_0)\cup \{v_1\}$, where $deg(v_1)=0$, and $E(G_1)=E(G_0)$,\\ then $t(G_0,F)=t(G_1,F)$.
\end{theorem}

\begin{proof}
Begin by observing that, for a fixed graph $F$ and arbitrary $G_0$, $t(G_0,F)=\frac{|H_0|}{|M_0|}$, where $M_0$ represents all possible mappings from $G_0$ into $F$ and $H_0$ represents those mappings that are also homomorphisms. Then $|M_0|={|V(F)|}^{|V(G_0)|}$.

Now construct $G_1$ by adding a single vertex $v_1$ to $G_0$ such that $deg(v_1)=0$, and observe that $|M_1|={|V(F)|}^{|V(G_1)|}={|V(F)|}^{|V(G_0)|+1}=|V(F)|\cdot |M_0|$. Further, since $E(G_0)=E(G_1)$, the only differences between $H_0$ and $H_1$ will be the choices of $\varphi(v_1)$ for any $\varphi\in H_1$. Since $deg(v_1)=0$, $\varphi(v_1)$ may be any of the $|V(F)|$ vertices of $F$. Therefore, $|H_1|=|V(F)|\cdot |H_0|$.

\noindent Thus, $t(G_1,F)=\frac{|H_1|}{|M_1|}=\frac{|V(F)|\cdot |H_0|}{|V(F)|\cdot |M_0|}=\frac{|H_0|}{|M_0|}=t(G_0,F)$.
\end{proof}
\noindent By virtue of Theorem 2.6, we may safely assume that our general graphs $G$ have no isolated vertices, thus simplifying the problem of counting homomorphisms from $G$ to $F$.

\section{Conclusion}
This paper began with the study of homomorphism densities between two graphs. We produced a set of inequalities that bound $t(G,F)$ when either $G$ or $F$ is a member of the special class of graphs known as complete graphs. We also showed that a homomorphism is guaranteed (i.e. $t(G,F)=1$) only when all vertices of $G$ are isolated. Additionally, we provided a theorem that allows us to simplify the difficult problem of exactly counting graph homomorphisms by ignoring any isolated vertices in $G$. Finally, this investigation led us to the notion of folded graphs, which have become our current focus.

\newpage
\bibliographystyle{plain}
\bibliography{mybib}
\end{document}